\documentclass[a4paper,fleqn%
]{amsart}
\usepackage{times,cite}
\theoremstyle{plain}

\newtheorem{theorem}{Theorem}[section]
\newtheorem{assumption}[theorem]{Assumption}
\newtheorem{proposition}[theorem]{Proposition}
\newtheorem{lemma}[theorem]{Lemma}
\newtheorem{corollary}[theorem]{Corollary}

\newtheorem*{theorem*}{Theorem}
\theoremstyle{definition}
\newtheorem{remark}[theorem]{Remark}
\newtheorem{definition}[theorem]{Definition}

\usepackage[]{graphicx}
\usepackage{a4wide}
\usepackage{pgfplots}
\usepackage[T1]{fontenc}
\usepackage[ansinew]{inputenc}
\usepackage{lmodern} 
\usepackage{graphicx}
\usepackage{amsmath}
\usepackage{amsthm}
\usepackage{amsfonts}
\usepackage{amssymb}
\usepackage{setspace}
\usepackage{mathrsfs}
\usepackage[all]{xy}
\usepackage{enumerate}

\def\xX{\underline{X}}

\def\TT{\underline{T}}

\def\T{\mathcal{T}}
\def\AA{\underline{A}}
\def\aA{\underline{A}}
\def\A{\mathcal{A}}

\def\dom{D}
\def\BC{\mathrm{C}_{\mathrm{b}}}
\def\BUC{\mathrm{UC}_{\mathrm{b}}}
\def\semis{\mathscr{P}}
\def\RR{\mathbb{R}}
\def\NN{\mathbb{N}}

\def\dd{\mathrm{d}}
\def\ee{\mathrm{e}}

\def\tlim{\mathop{\tau\mathrm{lim}}}
\def\Tlim{\mathop{\tau_{-1}\mathrm{lim}}}
\def\tsotlim{\mathop{\tau_{\mathrm{sot}}\mathrm{lim}}}

\def\LLL{\mathscr{L}}
\def\sot{\mathrm{sot}}

\def\Fav_#1{D_A(#1)}
\def\Hol_#1{D_{A,0}(#1)}
\def\XHol_#1{\underline{D}_{A,0}(#1)}
\newenvironment{abc}{\begin{enumerate}[{\rm (a)}]}{\end{enumerate}}
\newenvironment{num}{\begin{enumerate}[{\rm 1.}]}{\end{enumerate}}
\newenvironment{iiv}{\begin{enumerate}[{\rm (i)}]}{\end{enumerate}}

\def\Id{\mathrm{I}}

\begin{document}



\title{A Desch--Schappacher Perturbation Theorem for Bi-Continuous Semigroups}


\author{Christian Budde}%
 \email{cbudde@uni-wuppertal.de}
 %
\author{B\'{a}lint Farkas} 
\email{farkas@math.uni-wuppertal.de}
\address{University of Wuppertal, School of Mathematics and Natural Science, Gaussstrasse 20, 42119 Wuppertal, Germany}

\begin{abstract}
	We prove a  Desch--Schappacher  type perturbation theorem for one-parameter semigroups on Banach spaces which are not strongly continuous for the norm, but possess a weaker continuity property. In this paper we chose to work in the framework of bi-continuous semigroups. This choice has the advantage that we can treat in a unified manner two important classes of semigroups: implemented semigroups on the Banach algebra $\LLL(E)$ of bounded, linear operators on a Banach space $E$, and semigroups on the space of bounded and continuous functions over a Polish space induced by jointly continuous semiflows. For both of these classes  we present an application of our abstract perturbation theorem.%
%
%
%
\end{abstract}
\maketitle                   






\section*{Introduction}
\noindent As suggested by Greiner in \cite{Greiner1987} abstract perturbation theory of one-parameter semigroups provides good means to change the domain of a semigroup generator. For this an enlargement of the underlying Banach space may be necessary and extrapolation spaces become important. One of the well-known results in this direction goes back to the papers of Desch and Schappacher, see \cite{DeschSchappacher} and \cite{Desch1988}. Another prominent example of such general perturbation techniques is due to Staffans and Weiss, \cite{SW2004,Staffans2005}, and an elegant  abstract operator theoretic/algebraic approach has been developed by  Adler, Bombieri and Engel in \cite{WS}. A general theory of unbounded domain perturbations is given by Hadd, Manzo and Rhandi \cite{Rhandi2014}. A more recent paper  by B\'{a}tkai, Jacob, Voigt and Wintermayr \cite{BJVW2018} extends the notion of positivity to extrapolation spaces, and studies positive perturbations for positive semigroups on $\mathrm{AM}$-spaces. Hence, the study of abstract Desch--Schappacher type perturbations is a lively research field, to which we contribute with the present article. The reason for such an active interest in this area is that the range of application is vast. We mention here only a selection from the most recent ones: boundary perturbations by Nickel \cite{Nickel2004}, boundary feedback by Casarino, Engel, Nagel and Nickel \cite{CENN2003}, boundary control by Engel, Kramar Fijav\v{z}, Kl\"{o}ss, Nagel and Sikolya \cite{EKKNS2010} and Engel and Kramar Fijav\v{z} \cite{EK2017}, port-Hamiltonian systems by Baroun and Jacob \cite{BJ2009}, control theory by Jacob, Nabiullin, Partington and Schwenninger \cite{JNPS2018,JNPS2016} and Jacob, Schwenninger and Zwart \cite{JSZ2018} and vertex control in networks by Engel and Kramar Fijav\v{z} \cite{EK2008, EK2018}.

\medskip\noindent All the previously mentioned abstract perturbation results were developed for strongly continuous semigroups of linear operators on Banach spaces, $C_0$-semigroups for short. This is, for certain applications, e.g., for the theory of Markov transition semigroups, far too restrictive. For this situation the Banach space of bounded and continuous functions over a Polish space is the most adequate, but on this space the strong continuity with the respect to norm is, in general,  a too stringent requirement.

\medskip\noindent K\"uhnemund in \cite{KuPhD} has developed the abstract theory of  bi-continuous semigroups, which has the advantage that not only Markov transition semigroup, but also semigroups induced by jointly continuous flows or implemented semigroups, just to mention a few, can be handled in a unified manner. Some perturbation result for bi-continuous semigroups are known, see \cite{FaStud,FaSF,FaPHD}, however, none of which is suitable for domain perturbations.
 
\medskip\noindent As first step in longer a program this paper treats a Desch--Schappacher type perturbation theorem for this class of semigroups. Since the theory of bi-continuous semigroups uses a Banach space norm and a additional locally convex topology, it is fundamental to relate our results to the existing, analogous ones on locally convex spaces. We recall the next result, due to Jacob, Wegner, Wintermayr, from \cite{JWW}.

\begin{theorem*}
Let $X$ be a sequentially complete, locally convex space with fundamental system $\Gamma$ of continuous seminorms, and let $(A,\dom(A))$ be the generator of a locally equicontinuous $C_0$-semigroup $(T(t))_{t\geq0}$ on $X$. Moreover, let $\overline{X}$ be a sequentially complete locally convex space such that
\begin{abc}
	\item $X\subseteq\overline{X}$ is dense and the inclusion map is continuous,
	\item $\overline{A}$ with domain $\dom(\overline{A})=X$ generates a locally equicontinuous $C_0$-semigroup $(\overline{T}(t))_{t\geq0}$ on $\overline{X}$ such that $\overline{T}(t)_{|X}=T(t)$ holds for all $t\geq0$.
\end{abc}
Let $B:X\rightarrow\overline{X}$ be a linear and continuous operator and $t_0>0$ be a number such that
\begin{abc}\setcounter{enumi}{2}
	\item ${\displaystyle{\forall f\in\mathrm{C}\left(\left[0,t_0\right],X\right): \int_0^{t_0}{\overline{T}(t_0-t)Bf(t)\ \dd{t}}\in X}}$,
	\item ${\displaystyle{\forall p\in\Gamma\ \exists K\in\left(0,1\right)\ \forall f\in\mathrm{C}\left(\left[0,t_0\right],X\right): p\left(\int_0^{t_0}{\overline{T}(t_0-t)Bf(t)\ \dd{t}}\right)\leq K\cdot\sup_{t\in\left[0,t_0\right]}{p(f(t))}}}$.
\end{abc}
Then the operator $(C,\dom(C))$ defined by 
\[
Cx=(\overline{A}+B)x\ \text{for}\ x\in\dom(C)=\left\{x\in X: (\overline{A}+B)x\in X\right\}
\]
generates a locally equicontinuous $C_0$-semigroup on $X$ if and only if $\dom(C)\subseteq X$ is dense.
\end{theorem*}

\medskip
\noindent We will prove a similar result for bi-continuous semigroups with the advantage that we can relax condition (c) of the previous theorem in the sense that we allow different seminorms on the left- and the right-hand side of the inequality. Moreover, one has to change and expand the conditions for the bi-continuous case carefully to obtain a good interplay between the Banach space norm and the locally convex topology. A space $\overline{X}$ with the properties used in the theorem above is called an extrapolation space. For $C_0$-semigroups on Banach spaces the classical  construction is presented in \cite[Chapter II, Sect. 5a]{EN} in a self-contained manner. Extrapolation spaces for $C_0$-semigroups on locally convex spaces are constructed by Wegner in \cite{W}. Extrapolated bi-continuous semigroups and extrapolation spaces are recently treated by Budde and Farkas in \cite{BF}.

\medskip\noindent This paper is organized as follows. In the first section we recall some definitions and results for bi-continuous semigroups and give some preliminary constructions needed for the Desch--Schappacher perturbation result, which is stated and proved as Theorem \ref{thm:DS} in Section \ref{sec:DS}.  Section \ref{sec:adm} contains a sufficient condition for operators to satisfy the hypothesis of the abstract perturbation theorem, see Theorem \ref{thm:admDS}. In Section \ref{sec:trans} we prove that for a large class of bounded functions $g:\RR\to\mathbb{C}$ which are continuous up to a discrete set of jump discontinuities and for each bounded (complex) Borel measure $\mu$ on $\RR$ the operator
\[
Cf:=f'+\int_\RR f\,\dd\mu\cdot g
\]
with appropriate domain generates a bi-continuous semigroup on the Banach space $\BC(\RR)$ of bounded, continuous functions on $\RR$.
Section \ref{sec:impl} is devoted to Desch--Schappacher perturbations of left implemented semigroups on the Banach algebra $\LLL(E)$ of bounded, linear operators on a Banach space $E$.

\section{Preliminaries}\label{sec:Pre}
\subsection{Bi-continuous semigroups}\label{subsec:bicontsemi}                                     
The class of bi-continuous semigroups was introduced by K\"uhnemund in \cite{Ku} and \cite{KuPhD} to treat one parameter semigroups on Banach spaces that are not strongly continuous for the norm (i.e., not $C_0$-semigroups), but enjoy continuity properties with respect to a coarser, locally convex topology.
Here we briefly describe the basic ingredients needed for this theory.

\medskip Throughout this paper we need the following main assumptions.

\begin{assumption}\label{ass:bicontassump}
Consider a triple $(X,\left\|\cdot\right\|,\tau)$ where $X$ is a Banach space, and
\begin{abc}
	\item $\tau$ is a Hausdorff topology, coarser than the norm-topology on $X$.
	\item $\tau$ is sequentially complete on norm-bounded sets, i.e., every $\left\|\cdot\right\|$-bounded $\tau$-Cauchy sequence is $\tau$-convergent.
	\item The dual space of $(X,\tau)$ is norming for $X$, i.e.,
	\[
	\left\|x\right\|=\sup_{\substack{\varphi\in(X,\tau)'\\\left\|\varphi\right\|\leq1}}{\left|\varphi(x)\right|}\ \text{for all}\ x\in X.
	\]
\end{abc}
\end{assumption}

\begin{remark}
 One can reformulate the third assumption equivalently (see \cite[Rem. 4.2]{BF} and \cite[Lemma 4.4]{Kraaij2016}): There is a set $\semis$ of $\tau$-continuous seminorms defining the topology $\tau$, such that
\begin{equation}\label{eq:semisnorm}
\|x\|=\sup_{p\in\semis}p(x).
\end{equation}
\end{remark}

\begin{definition}\label{def:biequicont}
A family of operators $\mathcal{B}\subseteq\LLL(X)$ is called bi-equicontinuous if for each norm-bounded $\tau$-null sequence $(x_n)_{n\in\NN}$ in $X$ one has
\[
\tlim_{n\rightarrow\infty}{Bx_n}=0,
\]
uniformly for $B\in\mathcal{B}$.
\end{definition}

\begin{definition}\label{def:bicontdef}
Let $(X,\left\|\cdot\right\|,\tau)$ be a triple satisfying Assumption \ref{ass:bicontassump}. We call a family of bounded linear operators $(T(t))_{t\geq0}$ a $\tau$-\emph{bi-continuous semigroup} if
\begin{num}
\item $T(t+s)=T(t)T(s)$ and $T(0)=\Id$ for all $s,t\geq 0$ (semigroup law).
\item $(T(t))_{t\geq0}$ is strongly $\tau$-continuous, i.e., the map $\xi_x:[0,\infty)\to(X,\tau)$ defined by $\xi_x(t)=T(t)x$ is continuous for every $x\in X$.
\item $(T(t))_{t\geq0}$ is exponentially bounded, i.e., there exist $M\geq1$ and $\omega\in\RR$ such that $\left\|T(t)\right\|\leq M\ee^{\omega t}$ for each $t\geq0$.
\item $(T(t))_{t\geq0}$ is locally-bi-equicontinuous, i.e., the family $\left\{T(t):\ t\in\left[0,t_0\right]\right\}$ is bi-equicontinuous for each $t_0>0$.
\end{num}
\end{definition}
By saying that $(T(t))_{t\geq0}$ is a bi-continuous semigroup on a Banach space $X$ we implicitly assume that the in the background the triple  $(X,\left\|\cdot\right\|,\tau)$ satisfying Assumption \ref{ass:bicontassump} is fixed.
One of the prominent examples of bi-continuous semigroups is the translation (semi)group $(T(t))_{t\geq0}$ on the space $\BC(\RR)$ of bounded continuous functions on $\RR$,
\[
T(t)f(x)=f(x+t),\quad t\geq0,\:f\in \BC(\RR),\:x\in \RR.
\]
This semigroup is not strongly continuous with respect to the supremum-norm $\left\|\cdot\right\|_{\infty}$, but it becomes strongly continuous with respect to the compact-open topology $\tau_{\mathrm{co}}$, the locally convex topology on $\BC(\RR)$ induced by the family of seminorms $\semis=\{p_K:\ K\subseteq\RR\ \text{compact}\}$ where
\[
p_K(f)=\sup_{x\in K}{\left|f(x)\right|},\quad f\in\BC(\RR).
\] 
Similarly to the case of $C_0$-semigroups we define the generator for a bi-continuous semigroup as follows.

\begin{definition}\label{def:BiGen}
Let $(T(t))_{t\geq0}$ be a bi-continuous semigroup on $X$. The (infinitesimal) generator of $(T(t))_{t\geq0}$  is the linear operator $(A,\dom(A))$ defined by
\[Ax:=\tlim_{t\to0}{\frac{T(t)x-x}{t}}\] with domain 
\[\dom(A):=\Bigl\{x\in X:\ \tlim_{t\to0}{\frac{T(t)x-x}{t}}\ \text{exists and} \ \sup_{t\in(0,1]}{\frac{\|T(t)x-x\|}{t}}<\infty\Bigr\}.\]
\end{definition}

The following theorem summarizes the most essential properties of bi-continuous semigroups and their generators (see \cite{Ku},\cite{FaStud}).

\begin{theorem}\label{thm:bicontprop}
Let $(T(t))_{t\geq0}$ be a bi-continuous semigroup with generator $(A,\dom(A))$. Then the following hold:
\begin{abc}
\item $A$ is bi-closed, i.e., whenever $x_n\stackrel{\tau}{\to}x$ and $Ax_n\stackrel{\tau}{\to}y$ and both sequences are norm-bounded, then $y\in\dom(A)$ and $Ax=y$.
\item $\dom(A)$ is bi-dense in $X$, i.e., for each $x\in X$ there exists a norm-bounded sequence $(x_n)_{n\in\NN}$ in $\dom(A)$ such that $x_n\stackrel{\tau}{\to}x$.
\item For $x\in\dom(A)$ we have $T(t)x\in\dom(A)$ and $T(t)Ax=AT(t)x$ for all $t\geq0$.
\item For $t>0$ and $x\in X$ one has \begin{align}\int_0^t{T(s)x\ \dd s}\in\dom(A)\ \ \text{and}\ \ A\int_0^t{T(s)x\ \dd s}=T(t)x-x,\end{align} where the integral has to be understood as a $\tau$-Riemann integral.
\item For $\lambda>\omega_0(T)$ one has $\lambda\in\rho(A)$ (thus $A$ is closed) and for $x\in X$ holds: \begin{align}\label{eq:bicontlaplace}
R(\lambda,A)x=\int_0^{\infty}{\ee^{-\lambda s}T(s)x\ \dd s}\end{align} where the integral is a $\tau$-improper integral.
\end{abc}
\end{theorem}

\medskip\subsection{Extrapolation spaces}\label{subsec:Extrap}
In this section we recall some results concerning extrapolation spaces for bi-continuous semigroups from \cite{BF}. Throughout this section we assume without loss of generality that $0\in\rho(A)$. One of the most important ingredient for extrapolation spaces is the following proposition.

\begin{proposition}\label{prop:StrCont}
Let $(T(t))_{t\geq0}$ be a bi-continuous semigroup on $X$ with generator $(A,\dom(A))$. The subspace $\xX_0:=\overline{\dom(A)}^{\left\|\cdot\right\|}\subseteq X$ is $(T(t))_{t\geq0}$-invariant and $(\TT(t))_{t\geq0}:=(T(t)_{|X_0})_{t\geq0}$ is the $C_0$-semigroup on $\xX_0$ generated by the part of $A$ in $\xX_0$ (this generator is denoted by $\AA_0$). 
\end{proposition}

The classical construction of the extrapolation spaces $\xX_{-n}$ corresponding to the $C_0$-semigroup $(\TT(t))_{t\geq0}$ is summarized in \cite[Chapter II, Section 5a]{EN}. Recall from there that one obtains $\xX_{-1}$ as a completion of $\xX_0$ with respect to the $\left\|\cdot\right\|_{-1}$-norm defined by
\[
\left\|x\right\|_{-1}:=\left\|\AA^{-1}x\right\|,\quad x\in\xX_0.
\]
Notice that $\xX_0$ is dense in $\xX_{-1}$ and that $(\TT(t))_{t\geq0}$ extends by continuity to a $C_0$-semigroup $(\TT_{-1}(t))_{t\geq0}$ on $\xX_{-1}$ with generator $(\AA_{-1},\dom(\AA_{-1}))$, where $\dom(\AA_{-1})=\xX_0$. By repeating this construction one obtains the following chain of spaces
\[
\xX_0\stackrel{\AA_{-1}}{\hookrightarrow}\xX_{-1}\stackrel{\AA_{-2}}{\hookrightarrow}\xX_{-2}\rightarrow\cdots
\]
where all maps are continuous and dense inclusions. Notice that also 
\[
\xX_0\hookrightarrow X\hookrightarrow\xX_{-1}
\]
holds so that we can identify $X$, by the continuity of the inclusions, as a subspace of $\xX_{-1}$. We define the extrapolation space $X_{-1}$ for the bi-continuous semigroup $(T(t))_{t\geq0}$ by
\begin{align*}
X_{-1}:=\AA_{-2}(X).
\end{align*}
The norm on $X_{-1}$ is defined by $\|x\|_{-1}:=\|\aA^{-1}_{-2}\|$, the locally convex topology $\tau_{-1}$ on $X_{-1}$ comes from the family of seminorms $\semis_{-1}:=\left\{p_{-1}:\ p\in\semis\right\}$ where
\[
p_{-1}(x):=p(\aA_{-2}^{-1}x),\quad p\in\semis, x\in X.
\]
It was shown in \cite{BF} that $(T(t))_{t\geq0}$ extends to a $\tau_{-1}$-bi-continuous semigroup $(T_{-1}(t))_{t\geq0}$ on $X_{-1}$ and has a generator $A_{-1}$ with domain $\dom(A_{-1})=X$. The operator $A_{-1}:X\to X_{-1}$ is an isomorphism intertwining the semigroups $(T(t))_{t\geq0}$  and $(T_{-1}(t))_{t\geq0}$. If we want to stress the dependence of the extrapolation space $X_{-1}$ corresponding to the operator $(A,\dom(A))$ we write $X_{-1}(A)$. This will be used for the discussion of Desch--Schappacher perturbations. 

\begin{remark}\label{rem:SemiEstim}
	By construction $A_{-1}:(X,\tau)\to (X_{-1},\tau_{-1})$ is continuous, and actually an isomorphism. In particular, we have
\[
\forall p\in\semis\ \exists L>0\ \exists\gamma\in\semis_{-1} \forall x\in X:\ p(x)\leq L(\gamma(x)+\gamma(A_{-1}x)).
\]
\end{remark}

\subsection{Admissibility space}
In this subsection we fix some notation. Let $(T(t))_{t\geq0}$ be a $\tau$-bi-continuous semigroup on a Banach space $X$ with generator $(A,\dom(A))$, where $\tau$ is generates family of seminorms $\semis$. Furthermore, let $B\in\LLL(X,X_{-1})$ such that $B:(X,\tau)\rightarrow({X}_{-1},\tau_{-1})$ is a continuous linear operator and define for $t_0>0$ the following space:

\begin{align}\label{eqn:AdmSp}
\mathfrak{X}_{t_0}:=\begin{Bmatrix}
F:\left[0,t_0\right]\rightarrow\mathscr{L}(X): \tau\text{-}\text{strongly continuous, norm bounded}\\
\text{and} \left\{F(t):t\in\left[0,t_0\right]\right\}\ \text{is bi-equicontinuous}
\end{Bmatrix}.
\end{align}

\begin{remark}
In \cite[Lemma 3.2]{FaStud} it was shown that for $t_0>0$ the space $\mathfrak{X}_{t_0}$ is indeed a Banach space (and in particular a Banach algebra) with respect to the norm 
\[
\left\|F\right\|:=\sup\limits_{t\in\left[0,t_0\right]}{\left\|F(t)\right\|}.
\]
\end{remark}
 
For $F\in\mathfrak{X}_{t_0}$ and $t\in\left[0,t_0\right]$ we define the so-called \emph{(abstract) Volterra operator} $V_B$ on $\mathfrak{X}_{t_0}$ by

\begin{align}\label{eqn:Voltera}
(V_BF)(t)x:=\int_0^t{T_{-1}(t-r)BF(r)x\ \dd r}.
\end{align}

The integral has to be understood in the sense of a $\tau_{-1}$-Riemann integral. Notice that in general for $x\in X$ we have $(V_BF)(t)x\in X_{-1}$. For the formulation of our main result we need the following definition. 

\begin{definition}\label{def:adm}
{Let $B\in \LLL(X,X_{-1})$ such that also $B:(X,\tau)\rightarrow(X_{-1},\tau_{-1})$ is continuous. The operator $B$ is said to be \emph{admissible}, if there is  $t_0>0$ such that the following conditions are satisfied:
\begin{iiv}
	\item $V_BF(t)x\in X$ for all $t\in\left[0,t_0\right]$ and $x\in X$.
	\item $\mathrm{Ran}(V_B)\subseteq\mathfrak{X}_{t_0}$.
	\item $\left\|V_B\right\|<1$.
\end{iiv}}
The set of all admissible operators $B:(X,\tau)\rightarrow(X_{-1},\tau_{-1})$ will be denoted by $\mathcal{S}_{t_0}^{DS,\tau}$. We write $B\in\mathcal{S}_{t_0}^{DS,\tau}(T)$ whenever it is important to emphasize for which semigroup $(T(t))_{t\geq0}$ the operator $B$ is admissible.
\end{definition}

\section{An abstract Desch--Schappacher perturbation result}
\label{sec:DS}
This section contains the formulation of the Desch--Schappacher type perturbation result and its proof. 

\begin{theorem}\label{thm:DS}
Let $(A,\dom(A))$ be {the} generator of a $\tau$-bi-continuous semigroup $(T(t))_{t\geq0}$ on a Banach space $X$. Let $B:X\to X_{-1}$ such that  $B\in\mathcal{S}_{t_0}^{DS,\tau}$ for some $t_0>0$. Then the operator $(A_{-1}+B)_{|{X}}$ with domain
 \begin{align*}
	 \dom((A_{-1}+B)_{|X}):=\left\{x\in X:\ A_{-1}x+Bx\in X\right\}
	 \end{align*} 
	 generates a $\tau$-bi-continuous semigroup $(S(t))_{t\geq0}$ on $X$. Moreover, the semigroup $(S(t))_{t\geq0}$ satisfies the variation of parameters formula
\begin{align}\label{eqn:VariPara}
S(t)x=T(t)x+\int_0^t{T_{-1}(t-r)BS(r)x\ \dd{r}},
\end{align}
for every $t\geq0$ and $x\in X$.
\end{theorem}

\begin{proof}Since $\left\|V_B\right\|<1$ by hypothesis, we conclude that $1\in\rho(V_B)$ . Now let $t>0$ be arbitrary and write $t=nt_0+t_1$ for $n\in\mathbb{N}$ and $t_1\in\left[0,t_0\right)$. Define
\[
S(t):=((R(1,V_B)T_{|[0,t_0]})^n(t_0)\cdot(R(1,V_B)T_{|[0,t_0]})(t_1).
\]
We first show that $(S(t))_{t\geq0}$ is a semigroup. For $0\leq s,t\leq s+t\leq t_0$ and $n\in\NN$ we prove the following identity (cf. \cite[Chapter III, Sect. 3]{EN}) 
\begin{align}\label{eqn:IdentVB}
(V_B^nT)(t+s)=\sum_{k=0}^{n}{(V_B^{n-k}T_{|[0,t_0]})(s)\cdot(V_B^kT)(t)},\ \ \forall n\in\mathbb{N}
\end{align}
by induction. We abbreviate $V:=V_B$. Since $V^0=\Id$, equation \eqref{eqn:IdentVB} is trivially satisfied for $n=0$. Now assume that \eqref{eqn:IdentVB} is true for some $n\in\mathbb{N}$. Then we obtain by this hypothesis that
\begin{align*}
\sum_{k=0}^{n+1}&{(V^{n+1-k}T)(s)\cdot(V^kT)(t)}\\
&=\sum_{k=0}^n{\left(\int_0^s{T_{-1}(s-r)BV^{n-k}T(r)\ \dd r}\right)\cdot V^kT(t)}+T(s)\int_0^t{T_{-1}(t-r)BV^nT(r)\ \dd r}\\
&=\int_0^s{T_{-1}(s-r)B\sum_{k=0}^n{V^{n-k}T(r)\cdot V^kT(t)}\ \dd r}+\int_0^t{T_{-1}(s+t-r)BV^nT(r)\ \dd r}\\
&=\int_0^s{T_{-1}(s-r)B{V^{n}T(r+t)\ \dd r}}+\int_0^t{T_{-1}(s+t-r)BV^nT(r)\ \dd r}\\
&=\int_t^{s+t}{T_{-1}(s+t-r)B{V^{n}T(r)\ \dd r}}+\int_0^t{T_{-1}(s+t-r)BV^nT(r)\ \dd r}\\
&=V^{n+1}T(s+t).
\end{align*}
By  this we can conclude that $(S(t))_{t\geq0}$ satisfies the semigroup law for $0\leq s,t\leq s+t\leq t_0$. Indeed, for each $t\in\left[0,t_0\right]$ the point evaluation $\delta_t:\mathfrak{X}_{t_0}\rightarrow\LLL(X)$ is a contraction and since $\left\|V\right\|<1$ by hypothesis the inverse of $\Id-V$ is given by the Neumann series. Therefore, 
\[
S(t)=\delta_t\left(\sum_{n=0}^{\infty}{V^nT}\right)=\sum_{n=0}^{\infty}{(V^nT)(t)},\quad t\in\left[0,t_0\right].
\]
Moreover, we have
\[
\left\|(V^nT)(t)\right\|=\leq\left\|V^n\right\|\cdot\left\|T_{|[0,t_0]}\right\|,
\]
and we conclude that the series above converges absolutely. Hence
\begin{align*}
S(s)S(t)&=\sum_{n=0}^{\infty}{(V^nT)(s)}\cdot\sum_{n=0}^{\infty}{(V^nT)(t)}\\
&=\sum_{n=0}^{\infty}{\sum_{k=0}^{n}{(V^{n-k}T)(s)(V^kT)(t)}}\\
&=\sum_{n=0}^{\infty}{(V^nT)(s+t)}=S(s+t).
\end{align*}
Now we show that $S(t)S(s)=S(t+s)$ for all $t,s>0$. For that let $t,s>0$ be arbitrary and $n,m\in\mathbb{N}$ and $t_1,t_2\in\left[0,t_0\right)$ such that $t=nt_0+t_1$ and $s=mt_0+t_2$. Then we obtain the following
\begin{align*}
S(t)S(s)&=S(t_0)^nS(t_1)S(t_0)^mS(t_2)\\
&=S(t_0)^nS(t_0)^mS(t_1)S(t_2)\\
&=\begin{cases}S(t_0)^{n+m}S(t_1+t_2),&\text{if}\ t_1+t_2<t_0,\\S(t_0)^{n+m+1}S(t_2-(t_0-t_1)),&\text{if}\ t_1+t_2\geq t_0.\end{cases}
\end{align*}
But in both cases the right-hand side equals $S(t+s)$ by definition. Hence $(S(t))_{t\geq0}$ satisfies the semigroup law. The next step is to show that it is a $\tau$-bi-continuous semigroup. Notice that 
\[
S_{|\left[0,t_0\right]}(t)=R(1,V_B)T_{|\left[0,t_0\right]}(t)
\]
and hence $(S(t))_{t\geq0}$ is locally bounded and the set $\left\{S(t):t\in\left[0,t_0\right]\right\}$ is bi-equicontinuous. For $t>0$ let $m:=\big\lfloor \frac{t}{t_0}\big\rfloor$ and notice that $\left\{S(t_0)^k:\ 1\leq k\leq m\right\}$ is bi-equicontinuous, hence we conclude that the set 
\[
\left\{S(t_0)^k:\ 1\leq k\leq m\right\}\cdot\left\{S(s):\ s\in\left[0,t_0\right]\right\}
\]
 is also bi-equicontinuous. By definition of $(S(t))_{t\geq0}$ we obtain $\tau$-strong continuity, and hence $(S(t))_{t\geq0}$ is a $\tau$-bi-continuous semigroup. We now prove
\begin{align*}
S(t)x=T(t)x+\int_0^t{T_{-1}(t-r)BS(r)x\ \dd r}
\end{align*}
for each $t>0$ and $x\in X$ by proceeding similarly to \cite[Chapter III, Sect. 3]{EN}. For $t=nt_0+t_1$, $n\in\mathbb{N}$ and $t_1\in\left[0,t_0\right)$, we obtain:
\small
\begin{align*}
&\int_0^{t}{T_{-1}(t-r)BS(r)\ \dd r}\\
=&\sum_{k=0}^{n-1}{\int_{kt_0}^{(k+1)t_0}{T_{-1}(t-r)BS(r)\ \dd r}}+\int_{nt_0}^t{T_{-1}(t-r)BS(r)\ \dd r}\\
=&\sum_{k=0}^{n-1}{T_{-1}(t-(k+1)t_0)\int_{0}^{t_0}{T_{-1}(t_0-r)BS(r)\ \dd r}\cdot S(kt_0)}+\int_0^{t_1}{T_{-1}(t_1-r)BS(r)\ \dd r}\cdot S(nt_0)\\
=&\sum_{k=0}^{n-1}{T(t-(k+1)t_0)(S(t_0)-T(t_0))S(kt_0)}+(S(t_1)-T(t_1))S(nt_0)=S(t)-T(t).
\end{align*}

\normalsize
 The next step is to show that the resolvent set of $(A_{-1}+B)_{|X}$ is non-empty. For this we claim that $R(\lambda, A_{-1})B$ is bounded with $\left\|R(\lambda,A_{-1})B\right\|<1$ for $\lambda$ large enough. Choose $M\geq0$ and $\omega\in\mathbb{R}$ such that $\left\|T(t)\right\|\leq M\ee^{\omega t}$ for all $t>0$. Then for $\lambda>\omega$ we obtain:

\[
R(\lambda,A_{-1})B=\int_0^{\infty}{\ee^{\lambda r}T_{-1}(r)B\ \dd r}=\sum_{n=0}^{\infty}{\ee^{-\lambda nt_0}T(nt_0)(V_BF_{\lambda})(t_0)}
\]

where $F_{\lambda}(r):=\ee^{-\lambda(t_0-r)}\Id\in\mathfrak{X}_{t_0}$. From this we obtain the following estimate:
\[
\left\|R(\lambda,A_{-1})B\right\|\leq\left\|V_B\right\|+\frac{M\ee^{(\omega-\lambda)t_0}}{1-\ee^{(\omega-\lambda)t_0}}\left\|V_B\right\|.
\]

Since $\left\|V_B\right\|<1$ we conclude for sufficient large $\lambda$:

\[
\left\|R(\lambda,A_{-1})B\right\|<1.
\]

This yields $1\in\rho(R(\lambda,A_{-1})B)$ for large $\lambda$ and then invertibility of $\lambda-(A_{-1}+B)_{|X}$, since

\[
\lambda-(A_{-1}+B)_{|X}=(\lambda-A)(\Id-R(\lambda,A_{-1})B).
\]

Hence the resolvent set of $(A_{-1}+B)_{|X}$ contains each sufficiently large $\lambda$. In the last step we will show that $(A_{-1}+B)_{|X}$ is actually the generator of the $\tau$-bi-continuous semigroup $(S(t))_{t\geq0}$. Denote by $(C,\dom(C))$ the generator of $(S(t))_{t\geq0}$. Let $\lambda>\max{\left(\omega_0(T),\omega_0(S)\right)}$, then by the variation of constant formula \eqref{eqn:VariPara}, the resolvent representation as Laplace transform \cite[Lemma 7]{Ku} and the fact that we may interchange the improper $\tau$-Riemann integral and the {$\tau_{-1}$-Riemann} integral by an application of \cite[Lemma 1.7]{KuPhD} we obtain
\[
R(\lambda,C)=R(\lambda,A)+R(\lambda,A_{-1})BR(\lambda,C).
\]
Whence we conclude
\[
(\Id-R(\lambda,A_{-1})B)R(\lambda,C)=R(\lambda,A),
\]
and therefore
\[
\Id=(\lambda-A)(\Id-R(\lambda,A_{-1})B)R(\lambda,C)=(\lambda-(A_{-1}+B)_{|X})R(\lambda,C).
\]
It follows that $C\subseteq (A_{-1}+B)_{|X}$ and by the previous observations $C=(A_{-1}+B)_{|X}$.
\end{proof}

\medskip\subsection{Abstract Favard Spaces and Comparison}

We recall the definition of (abstract) Favard spaces from  \cite[Chapter III, Sect. 5b]{EN} and \cite{BF}. Let $(A,\dom(A))$ be an operator with \emph{ray of minimal growth}, i.e.,  $(0,\infty)\subseteq\rho(A)$ and for some $M\geq 0$
\begin{equation}\label{eq:weakHY}
\|\lambda R(\lambda,A)\|\leq M\quad\text{for all $\lambda>0$}.
\end{equation}
For $\alpha\in\left(0,1\right]$ the (abstract) Favard space $F_{\alpha}(A)$ is defined by
\[
F_{\alpha}(A):=\left\{x\in X:\ \sup_{\lambda>0}\|\lambda^\alpha AR(\lambda,A)x\|<\infty\right\}.
\]
If in addition $(A,\dom(A))$ is the generator of a {($\tau$-bi-continuous)} semigroup $(T(t))_{t\geq0}$, then
\[
F_{\alpha}(A)=\left\{x\in X_0:\ \sup_{s\in\left(0,1\right)}\frac{\|T(s)x-x\|}{s^{\alpha}}<\infty\right\}=:F_{\alpha}(T).
\]

The space $F_0(A)$ defined by
\[
F_0(A):=F_1(A_{-1}),
\]
is called the extrapolated Favard class. The {restricted operator} $A_{-1}|_{F_0(A)}:F_0(A)\rightarrow F_1(A)$ is an isometric isomorphism. Moreover, if $(A,\dom(A))$ generates a {$C_0$-semigroup} $(T(t))_{t\geq0}$ it is shown in \cite{nagel1993inhomogeneous} that also
\[
F_0(T)=F_1(T_{-1})
\]
holds. {In the next proposition is we show that Desch--Schappacher perturbations of bi-continuous semigroups, which satisfy a special range condition concerning the extrapolated Favard class, gives us semigroups which are close to each other in some sense.}
\begin{proposition}\label{prop:ImplDS2}
Let $(T(t))_{t\geq0}$ be a $\tau$-bi-continuous semigroup on $X$ generated by $(A,\dom(A))$. Suppose that $B\in\mathcal{S}^{DS,\tau}_{t_0}$ with $\mathrm{Ran}(B)\subseteq F_0(A)$ and let $(S(t))_{t\geq0}$ be the perturbed semigroup. Then there exists $C\geq0$ such that for each $t\in\left[0,1\right]$ one has
\[
\left\|T(t)-S(t)\right\|\leq Ct.
\]
\end{proposition}

\begin{proof} {We may assume $\omega_0(T)<0$.}
We find $M\geq0$ such that $\left\|T(t)\right\|\leq M$ and $\left\|S(t)\right\|\leq M$ for every $t\in\left[0,1\right]$. Since $\mathrm{Ran}(B)\subseteq F_0(A)$ we conclude that $A_{-1}^{-1}B:X\rightarrow F_1(A)$. Hence $A_{-1}^{-1}B$ is bounded by the closed graph theorem and we find $K\geq0$ such that $\left\|A_{-1}^{-1}Bx\right\|_{F_1(A)}\leq K\left\|x\right\|$ for each $x\in X$. Let $\semis_{-1}$ the family of seminorms corresponding to the first extrapolation space (see Section \ref{subsec:Extrap}). By using \eqref{eqn:VariPara} we obtain
\begin{align*}
\left\|S(t)x-T(t)x\right\|&=\left\|A_{-1}\int_0^t{T(t-r)A_{-1}^{-1}BS(r)x\ \dd{r}}\right\|\\
&=\left\|\Tlim_{h\rightarrow0}{\frac{T_{-1}(h)-\Id}{h}}\int_0^t{T(t-r)A_{-1}^{-1}BS(r)x\ \dd{r}}\right\|\\
&=\left\|\Tlim_{h\rightarrow0}\int_0^t{\frac{T(h)-\Id}{h}T(t-r)A_{-1}^{-1}BS(r)x\ \dd{r}}\right\|\\
&=\sup_{p\in\semis_{-1}}{\lim_{h\rightarrow0}p\left(\int_0^t{\frac{T(h)-\Id}{h}T(t-r)A_{-1}^{-1}BS(r)x\ \dd{r}}\right)}\\
&\leq\sup_{p\in\semis_{-1}}{\lim_{h\rightarrow0}\int_0^t{p\left(\frac{T(h)-\Id}{h}T(t-r)A_{-1}^{-1}BS(r)x\right)\ \dd{r}}}\\
&\leq\limsup_{h\rightarrow0}{\int_0^t{\left\|\frac{T(h)-\Id}{h}T(t-r)A_{-1}^{-1}BS(r)x\right\|\ \dd{r}}}\\
&\leq M\int_0^t{\left\|A_{-1}^{-1}BS(r)x\right\|_{F_1(A)}\ \dd{r}}\\
&\leq tKM^2\cdot\left\|x\right\|
\end{align*} 
for each $x\in X$ and $t\in\left[0,1\right]$. 
\end{proof}

\begin{corollary}
Let $(T(t))_{t\geq0}$ be a $\tau$-bi-continuous semigroup on $X$ generated by $(A,\dom(A))$. If $B\in\mathcal{S}_{t_0}^{DS,\tau}$ and $\mathrm{Ran}(B)\subseteq F_0(A)$, then the perturbed semigroup $(S(t))_{t\geq0}$ leaves the space of strong continuity $\xX_0:=\overline{\dom(A)}^{\left\|\cdot\right\|}$ invariant.
\end{corollary}
 
\section{Admissible operators}

\label{sec:adm}
Next we consider a sufficient condition for $B:(X,\tau)\rightarrow(X_{-1},\tau_{-1})$ to be admissible. Throughout this section we denote the space of continuous functions $f:\left[0,t_0\right]\rightarrow(X,\tau)$ which are $\left\|\cdot\right\|$-bounded by $\BC\left(\left[0,t_0\right],(X,\tau)\right)$. If equipped with the sup-norm, $\BC\left(\left[0,t_0\right],(X,\tau)\right)$ becomes a Banach space.

\begin{theorem}\label{thm:admDS}
{Let $(T(t))_{t\geq0}$ be a $\tau$-bi-continuous semigroup  with  generator $(A,\dom(A))$ on a Banach space $X$.} Let $\mathscr{P}$ be the set of generating continuous seminorms corresponding to $\tau$. Let {$B\in \LLL(X,X_{-1})$} such that $B:(X,\tau)\rightarrow(X_{-1},\tau_{-1})$ is a linear and continuous operator,  and let $t_0>0$ be such that
\begin{abc}
	\item $\displaystyle{\int\limits_0^{t_0}{T_{-1}(t_0-r)Bf(r)\ \dd r}\in X}$ for each $f\in\BC\left(\left[0,t_0\right],(X,\tau)\right)$.
	\item For every $\varepsilon>0$ and every $p\in\semis$ there exists $q\in\semis$ and $K>0$ such that for all $f\in\BC\left(\left[0,t_0\right],(X,\tau)\right)$
	\begin{align}
	p\left(\int_0^{t_0}{T_{-1}(t_0-r)Bf(r)\ \dd r}\right)\leq K\cdot\sup_{r\in\left[0,t_0\right]}{\left|q(f(r))\right|}+\varepsilon\left\|f\right\|_{\infty}.
	\end{align}
	\item There exists $M\in\left(0,\frac{1}{2}\right)$ such that for all $f\in\BC\left(\left[0,t_0\right],(X,\tau)\right)$ 
	\begin{align}
	\left\|\int\limits_0^{t_0}{T_{-1}(t_0-r)Bf(r)\ \dd r}\right\|\leq M\left\|f\right\|_{\infty}.
	\end{align}
\end{abc}

\noindent Then $B\in\mathcal{S}^{DS,\tau}_{t_0}$, and as a consequence the operator $(A_{-1}+B)_{|X}$ defined on the domain 
\[
\dom((A_{-1}+B)_{|X}):=\left\{x\in X:\ A_{-1}x+Bx\in X\right\}
\]
generates a $\tau$-bi-continuous semigroup.
\end{theorem}

\begin{proof}We first show $\mathrm{Ran}(V_B)\subseteq\mathfrak{X}_{t_0}$. Let $f\in\BC\left(\left[0,t_0\right],(X,\tau)\right)$ and define for $t\in\left[0,t_0\right]$ the auxiliary function $f_t:\left[0,t_0\right]\rightarrow X$ by
\[
f_t(r):=
\begin{cases}
f(0),&\ r\in\left[0,t_0-t\right],\\
f(r+t-t_0),&\ r\in\left[t_0-t,t_0\right].
\end{cases}
\]
Then $f_t\in\BC\left(\left[0,t_0\right],(X,\tau)\right)$ and
\begin{align}
\int_0^t{T_{-1}(t-r)Bf(r)\ \dd r}=\int_0^{t_0}{T_{-1}(t_0-r)Bf_t(r)\ \dd{r}}-\int_t^{t_0}{T_{-1}(r)Bf(0)\ \dd{r}}.
\end{align}
By Theorem \ref{thm:bicontprop}
\[
\int_t^{t_0}{T_{-1}(r)Bf(0)\ \dd{r}}=T(t)\int_0^{t_0-t}{T_{-1}(r)Bf(0)\ \dd{r}}\in\dom(A_{-1})=X.
\]
We conclude that the map {$\psi:\left[0,t_0\right]\rightarrow X_{-1}$} defined by
\begin{align}
\psi(t):=\int_0^t{T_{-1}(t-r)Bf(r)\ \dd r}
\end{align}
has values in $X$. Moreover, for $\varepsilon>0$ and $p\in\semis$ we have the following estimate:

\begin{align*}
&p(\psi(t)-\psi(s))\\
=&p\left(\int_0^t{T_{-1}(t-r)Bf(r)\ \dd r}-\int_0^s{T_{-1}(s-r)Bf(r)\ \dd r}\right)\\
\leq& p\left(\int_0^{t_0}{T_{-1}(t_0-r)B(f_t(r)-f_s(r))\ \dd r}\right)+p\left(\int_s^t{T_{-1}(r)Bf(0)\ \dd r}\right)\\
\leq& K\cdot\sup_{r\in\left[0,t_0\right]}{q(f_t(r)-f_s(r))}+p\left(\int_s^t{T_{-1}(r)Bf(0)\ \dd r}\right)+\varepsilon\left\|f_t-f_s\right\|_{\infty}\\
\leq& K\cdot\sup_{r\in\left[0,t_0\right]}{q(f_t(r)-f_s(r))}\\
&+L\cdot\left(\gamma\left(\int_s^t{T_{-1}(r)Bf(0)\ \dd r}\right)+\gamma\left((T_{-1}(t)-T_{-1}(s))Bf(0)\right)\right)+\varepsilon\left\|f_t-f_s\right\|\\
\leq& K\cdot\sup_{r\in\left[0,t_0\right]}{\left|q(f_t(r)-f_s(r))\right|}\\
&+L\cdot\left(\int_s^t{\gamma(T_{-1}(r)Bf(0))\ \dd r}+\gamma\left((T_{-1}(t)-T_{-1}(s))Bf(0)\right)\right)+2\varepsilon\left\|f\right\|_{\infty}
\end{align*}

where the $\gamma\in\semis_{-1}$ of the second to last inequality comes from Remark \ref{rem:SemiEstim}. The extrapolated semigroup $(T_{-1}(t))_{t\geq0}$ is strongly $\tau_{-1}$-continuous and $\gamma\in\semis_{-1}$, so that we can find $\delta_1>0$ such that 
\[
\gamma(T_{-1}(t)-T_{-1}(s)Bf(0))<\varepsilon\ \text{whenever}\ \left|t-s\right|<\delta_1.
\]
Moreover, $f$ is $\tau$-continuous and therefore uniformly $\tau$-continuous on compact sets, which gives us $\delta_2>0$ such that 
\[
\sup_{r\in\left[0,t_0\right]}{\left|q(f_t(r)-f_s(r))\right|}<\varepsilon\ \text{if}\ \left|t-s\right|<\delta_2.
\]
Last but not least, $\gamma(T_{-1}(r)Bf(0))$ is bounded by some constant $M>0$, so for $\delta_3=\frac{\varepsilon}{M}$ we have 
\[
\left|s-t\right|<\delta_3\ \Longrightarrow\ \int_s^t{\gamma(T_{-1}(r)Bf(0))\ \dd r}<\varepsilon.
\]
Now, we take $\delta:=\min\left\{\delta_1,\delta_2,\delta_3\right\}$ and obtain
\[
p(\psi(t)-\psi(s))<(K+2\left\|f\right\|_{\infty}+2L)\varepsilon,
\]
showing that $\psi:\left[0,t_0\right]\rightarrow X$ is $\tau$-continuous.

Next, we prove the norm-boundedness using the same techniques and arguments as in \cite[Chapter III, Sect. 3]{EN}. Let $f\in\BC\left(\left[0,t_0\right],(X,\tau)\right)$ and write \[f=\widetilde{f}_{\delta}+h_{\delta},\] where
\begin{align*}
h_{\delta}(x):=\begin{cases}
\left(1-\frac{r}{\delta}\right)f(0),&\quad 0\leq r<\delta,\\
0,&\quad \delta\leq r\leq t_0
\end{cases}
\end{align*}
for some $\delta>0$. Then $\widetilde{f}_{\delta}$ and $h_{\delta}$ are norm-bounded and continuous with respect to $\tau$, $\widetilde{f}_{\delta}(0)=0$ and $\|\widetilde{f}_{\delta}\|_{\infty}\leq2\|f\|_{\infty}$. Now we obtain
\begin{align*}
&\left\|\int_0^t{T_{-1}(t-r)Bf(r)\ \dd{r}}\right\|\leq\left\|\int_0^t{T_{-1}(t-r)B\widetilde{f}_{\delta}(r)\ \dd{r}}\right\|+\left\|\int_0^t{T_{-1}(t-r)Bh_{\delta}(r)\ \dd{r}}\right\|\\
\leq& M\left\|\widetilde{f}_{\delta}\right\|_{\infty}+K\left(\left\|\int_0^t{T_{-1}(t-r)Bh_{\delta}(r)\ \dd{r}}\right\|_{-1}+\left\|A_{-1}\int_0^t{T_{-1}(t-r)Bh_{\delta}(r)\ \dd{r}}\right\|_{-1}\right)\\
\leq& M\left\|\widetilde{f}_{\delta}\right\|_{\infty}+K\left\|\int_0^{\delta}{T_{-1}(t-r)\left(1-\frac{r}{\delta}\right)Bf(0)\ \dd{r}}\right\|_{-1}\\
&+K\left\|T_{-1}(t)Bf(0)-\frac{1}{\delta}\int_0^{\delta}{T_{-1}(t-r)Bf(0)\ \dd{r}}\right\|_{-1}.
\end{align*}
By taking $\delta\searrow0$ we obtain
\begin{align}\label{eqn:VBnorm}
\left\|\int_0^t{T_{-1}(t-r)Bf(r)\ \dd{r}}\right\|\leq 2M\left\|f\right\|_{\infty}.
\end{align}

We proceed with showing local bi-equicontinuity. For that let $(x_n)_{n\in\mathbb{N}}$ be a norm-bounded $\tau$-null-sequence. Let $\varepsilon>0$ and $p\in\mathscr{P}$, then by taking $f^n(r)=f(r)x_n$, we can find $q\in\semis$ such that

\begin{align*}
p\left(V_BF(t)x_n\right)=&p\left(\int_0^{t}{T_{-1}(t-r)BF(r)x_n\ \dd r}\right)\\
\leq&p\left(\int_0^{t_0}{T_{-1}(t_0-r)Bf^n(r)\ \dd r}-\int_t^{t_0}{T_{-1}(r)Bf^n(0)\ \dd r}\right)\\
\leq& K\cdot\sup_{r\in\left[0,t_0\right]}{\left|q(f^n(r))\right|}+p\left(\int_t^{t_0}{T_{-1}(r)Bf^n(0)\ \dd r}\right)+\varepsilon\left\|f^n_t\right\|\\
\leq& K\cdot\sup_{r\in\left[0,t_0\right]}{\left|q(f^n(r))\right|}+\varepsilon\left\|f^n_t\right\|\\
&+L\cdot\left(\gamma\left(\int_t^{t_0}{T_{-1}(r)Bf^n(0)\ \dd r}\right)+\gamma\left(\left(T_{-1}(t_0)-T_{-1}(t)\right)Bf^n(0)\right)\right).
\end{align*}

Now we can argue by the local bi-equicontinuity of $(T(t))_{t\geq0}$ and $(T_{-1}(t))_{t\geq0}$ and with the arbitrarily small $\varepsilon>0$ to conclude the {local} bi-equicontinuity of $V_BF$. Hence we see that $V_B$ maps $\mathfrak{X}_{t_0}$ to $\mathfrak{X}_{t_0}$ and by \eqref{eqn:VBnorm} that $\left\|V_B\right\|<1$ since by assumption $M\in\left(0,\frac{1}{2}\right)$.
\end{proof}

\section{{Perturbations of the translation semigroup}}\label{sec:trans}

Take $X=\BC(\RR)$ and let $(T(t))_{t\geq0}$ be the translation semigroup defined by $T(t)f(x)=f(x+t)$ (see also Section \ref{sec:Pre}). As already mentioned in Section \ref{subsec:bicontsemi}, this semigroup is $\tau_{\mathrm{co}}$-bi-continuous. Moreover, the resulting extrapolation spaces, in the notation we used in Section \ref{subsec:Extrap}, are given by (see \cite{BF}):
\begin{align*}
\xX_{-1}&=\left\{F=f-Df: f\in\BUC(\RR)\right\}\\
X_{-1}&=\left\{F=f-Df:  f\in\BC(\RR)\right\},
\end{align*}
where $\BUC(\RR)$ denotes the space of bounded uniformly continuous functions and $Df$ the distributional derivative of $f$. {The generator of $(T(t))_{t\geq0}$ is $A=D$ with domain $\dom(A):=\BC^1(\RR)$, and also $A_{-1}=D$ with domain $\dom(A_{-1})=\BC(\RR)$.  The extrapolated semigroup $(T_{-1}(t))_{t\geq 0}$ is the restriction to $X_{-1}$ of the left translation semigroup on the space $\mathscr{D}'(\RR)$ of distributions.}

Consider the function $g:\RR\rightarrow\RR$ defined by
\begin{align}  
g(x)=
\begin{cases}\label{eqn:gExtr}
0,& \ x\leq-1,\ x>1,\\
x,& -1<x\leq0,\\
2-x,& 0<x\leq1.\\
\end{cases}
\end{align}

 Notice that $g\in X_{-1}$, since $g=h-Dh$ where $h$ is the tent function on the real line defined by
\[
h(x)=
\begin{cases}
0,& x\leq-1,\ x>1,\\
x+1,& -1<x\leq0,\\
-x+1,& 0<x\leq1.
\end{cases}
\]
Let $\mu$ be a bounded regular Borel measure on $\RR$ and define the continuous functional $\Phi:\BC(\RR)\rightarrow\RR$ by $\Phi(f)=\int_{\RR}{f\ \dd\mu}$ and the operator $B:X\rightarrow X_{-1}$ by
\[
Bf:=\Phi(f)g.
\]
{This operator $B$ is by construction continuous with respect to the local convex topologies on the spaces $X$ and $X_{-1}$, {and also for the norms}. Moreover, $B$ has all properties required in Theorem \ref{thm:admDS}. To see this let $f\in\BC\left(\left[0,t_0\right],(X,\tau)\right)$ be arbitrary. Define a map $\psi:\RR\rightarrow\RR$ by
\[
\psi(\cdot)=\int_{0}^{t_0}{T_{-1}(t_0-r)Bf(r)(\cdot)\ \dd{r}}.
\]}
Observe that
\[
T_{-1}(t_0-r)Bf(r)(x)=T_{-1}(t_0-r)\Phi(f(r))g(x)=\Phi(f(r))g(x+t_0-r).
\]
{We claim that $\psi$ is continuous. Indeed, let  $\varepsilon>0$  be arbitrary, and notice that by substitution for each $x\in \RR$
\[
\int_0^{t_0}{\Phi(f(r))g(x+t_0-r)\ \dd{r}}=\int_{x}^{x+t_0}{\Phi(f(x+t_0-s))g(s)\ \dd{s}}.
\]
After this substitution we can make the following calculation for each $x,y\in \RR$
\begin{align*}
\psi(x)-\psi(y)&=\int_{x}^{x+t_0}{\Phi(f(x+t_0-s))g(s)\ \dd{s}}-\int_{y}^{y+t_0}{\Phi(f(x+t_0-s))g(s)\ \dd{s}}\\
&=\int_{0}^{x+t_0}{\Phi(f(x+t_0-s))g(s)\ \dd{s}}-\int_{0}^{x}{\Phi(f(x+t_0-s))g(s)\ \dd{s}}\\
&\quad-\int_{0}^{y+t_0}{\Phi(f(y+t_0-s))g(s)\ \dd{s}}+\int_{0}^{y}{\Phi(f(y+t_0-s))g(s)\ \dd{s}}\\
&=\int_{y+t_0}^{x+t_0}{\left(\Phi(f(x+t_0-s)-f(y+t_0-s))\right)g(s)\ \dd{s}}\\
&\quad +\int_{x}^{y}{\left(\Phi(f(x+t_0-s)-f(y+t_0-s))\right)g(s)\ \dd{s}}.
\end{align*}}
By the assumptions there exists $M>0$ such that
\[
\left\|\left(\Phi(f(x+t_0-\cdot)-f(y+t_0-\cdot))\right)g(\cdot)\right\|_{\infty}\leq M.
\]
For $\delta:=\frac{\varepsilon}{2M}>0$ and for $x,y\in \RR$ with  $\left|x-y\right|<\delta$  we have
\begin{align*}
&\left|\psi(x)-\psi(y)\right|\\
\leq&\int_{y+t_0}^{x+t_0}{\left|\left(\Phi(f(x+t_0-s)-f(y+t_0-s))\right)g(s)\right|\ \dd{s}}\\
&+\int_{x}^{y}{\left|\left(\Phi(f(x+t_0-s)-f(y+t_0-s))\right)g(s)\right|\ \dd{s}}\\
\leq&2\left|x-y\right|\cdot\left\|\left(\Phi(f(x+t_0-\cdot)-f(y+t_0-\cdot))\right)g(\cdot)\right\|_{\infty}\\
\leq&2M\cdot\left|x-y\right|<\varepsilon.
\end{align*}
This proves that $\psi\in\BC(\RR)$.

\medskip
Observe that in general we only have
\[
Q:=\int_{0}^{t_0}{T_{-1}(t_0-r)Bf(r)\ \dd{r}}\in X_{-1},
\]
so that point evaluation of this expression at $x\in \RR$ does not make sense. {We know, however,  that $\psi\in\BC(\RR)$, and that the pointwise Riemann-sums $R_n(x)$ for the integral
\[
\int_{0}^{t_0}\Phi(f(r))g(x+t_0-r) \dd{r}
\]
 converges for all $x\in \RR$ to $\psi(x)$.} If we can show that the sequence $(R_n)_{n\in\NN}$ converges in the sense of distributions we can conclude that $Q:=\int_{0}^{t_0}{T_{-1}(t_0-r)Bf(r)\ \dd{r}}=\psi\in X$. 
Let $\widetilde{\psi}\in\mathscr{D}(\RR)$ be a test function and define $\varphi:=\widetilde{\psi}-\mathscr{D}\widetilde{\psi}$. Then
\[
\left\langle (1-A_{-1})^{-1}R_n,\varphi\right\rangle\rightarrow\left\langle(1-A_{-1})^{-1}Q,\varphi\right\rangle.
\]
By the meaning of this pairing we conclude that $\langle R_n,\widetilde{\psi}\rangle\rightarrow\langle Q,\widetilde{\psi}\rangle$. By the above we conclude that $Q\in X$. 

\medskip
The next step is to estimate the norm. Notice that
\begin{align*}
\Bigl\|\int_{0}^{t_0}&{T_{-1}(t_0-r)Bf(r)\ \dd{r}}\Bigr\|_{\infty}
=\sup_{x\in\RR}{\Bigl|\int_{0}^{t_0}{\Phi(f(r))g(x+t_0-r)\ \dd{r}}\Bigr|}\\
&\leq\sup_{x\in\RR}{\int_0^{t_0}{\left|\Phi(f(r))\right|\cdot\left|g(x+t_0-r)\right| \dd{r}}}
\leq2{\int_0^{t_0}{\left|\Phi(f(r))\right|\ \dd{r}}}\\
&\leq2{\int_0^{t_0}{\int_{\RR}{\left|f(r)(x)\right|\ \dd\left|\mu\right|(x)}\ \dd{r}}}
\leq2\left|\mu\right|(\RR){\int_0^{t_0}{\left\|f(r)\right\|_{\infty}}\ \dd{r}}\\
&=2\left|\mu\right|(\RR){\int_0^{t_0}{\left\|f(r)\right\|_{\infty}}\ \dd{r}}
\leq2\left|\mu\right|(\RR)t_0\left\|f\right\|_{\infty}.
\end{align*}

In particular we can choose $t_0$ so small that $M:=2\left|\mu\right|(\RR)t_0<\frac{1}{2}$. Hence condition (c) of Theorem \ref{thm:admDS} is fulfilled. Condition (b) from Theorem \ref{thm:admDS} can be proven similarly. Let $K\subseteq\RR$ be an arbitrary compact set and $\varepsilon>0$. Then

\begin{align*}
p_K\left(\int_0^{t_0}{T_{-1}(t_0-r)Bf(r)(x)\ \dd{r}}\right)&\leq\sup_{x\in K}{\int_0^{t_0}{\left|\Phi(f(r))\right|\cdot\left|g(x+t_0-r)\right| \dd{r}}}\\
&\leq2\sup_{x\in K}{\int_0^{t_0}{\left|\Phi(f(r))\right|\ \dd{r}}}\\
&\leq2t_0|\mu|(\RR)\sup_{r\in\left[0,t_0\right]}{\sup_{y\in K'}}{\left|f(r)(y)\right|}+\varepsilon\left\|f\right\|_{\infty},
\end{align*}
{since by the regularity of the measure $\mu$ we choose $K'\subseteq\RR$ such that $\left|\mu\right|(\RR\setminus K')<\varepsilon$.} By Theorem \ref{thm:admDS} we conclude that $(A_{-1}+B)_{|X}$ generates again a $\tau_{\mathrm{co}}$-bi-continuous semigroup on $\BC(\RR)$. We now give an expression for the generator. Observe that $f\in\dom((A_{-1}+B)_{|X})$ if and only if $f\in\BC(\RR)$ and $f'+\Phi(f)g\in\BC(\RR)$ and this is precisely then, when the following conditions are satisfied.

\begin{align}\label{eqn:cont1}
\begin{cases}
\displaystyle{\lim_{t\nearrow-1}{\left(f'(t)+\Phi(f)g(t)\right)}=\lim_{t\searrow-1}{\left(f'(t)+\Phi(f)g(t)\right)}},\\
\displaystyle{\lim_{t\nearrow0}{\left(f'(t)+\Phi(f)g(t)\right)}=\lim_{t\searrow0}{\left(f'(t)+\Phi(f)g(t)\right)}},\\
\displaystyle{\lim_{t\nearrow1}{\left(f'(t)+\Phi(f)g(t)\right)}=\lim_{t\searrow1}{\left(f'(t)+\Phi(f)g(t)\right)}}.
\end{cases}
\end{align}

By the explicit expression for $g:\RR\rightarrow\RR$ we can rewrite Equation \eqref{eqn:cont1} as follows:
\begin{align}
\begin{cases}
\displaystyle{\lim_{t\nearrow-1}{f'(t)}=\lim_{t\searrow-1}{f'(t)-\Phi(f)}},\\
\displaystyle{\lim_{t\nearrow0}{f'(t)}=\lim_{t\searrow0}{f'(t)+2\Phi(f)}},\\
\displaystyle{\lim_{t\nearrow1}{f'(t)+\Phi(f)}=\lim_{t\searrow1}{f'(t)}}.
\end{cases}
\end{align}

Or equivalently

\begin{align}\label{eqn:condDom}
\displaystyle{\lim_{t\nearrow-1}{f'(t)}-\lim_{t\searrow-1}{f'(t)=-\frac{1}{2}\left(\lim_{t\nearrow0}{f'(t)}-\lim_{t\searrow0}{f'(t)}\right)=\lim_{t\nearrow1}{f'(t)-\lim_{t\searrow1}{f'(t)}}}=-\Phi(f)}.
\end{align}

We see that the generator $(C,\dom(C))$ of the perturbed semigroup is given by

\begin{align*}
Cf&=f'+\int_{\RR}{f\ \dd{\mu}}\cdot g,\quad f\in\dom(C),\\
\dom(C)&=\left\{f\in\BC(\RR):\ f\in\BC^1(\RR\setminus\left\{-1,0,1\right\})\ \text{and}\ \eqref{eqn:condDom}\ \text{holds} \right\}.
\end{align*}

The previous example uses a function $g\in X_{-1}$ which has three points of discontinuity, {with one sided limits at each of these points}. We generalize this to a countable (discrete) set of {jump} discontinuities. For that assume that $g\in X_{-1}$ is a function such that $\left\|g\right\|_{\infty}<\infty$ and that the set of discontinuities of $g$ is discrete. One defines again an operator $B:X\rightarrow X_{-1}$ by
\[
Bf:=\Phi(f)g:=\int_{\RR}{f\ \dd\mu}\cdot g,\quad f\in\BC(\RR).
\]
{Notice that none of previous calculations and arguments depend on the number of discontinuities (in fact, we only used that $g$ is bounded).} So we can conclude that $(A_{-1}+B)_{|X}$ generates a $\tau_{\mathrm{co}}$-bi-continuous semigroup on $X$. The only issue we have to care about are the conditions mentioned in \eqref{eqn:condDom}, that is an ``explicit'' description of the domain. { Let $Z:=\left\{x_1,x_2,x_3,\ldots\right\}$ be the set of discontinuities of $g$ that is assumed to be discrete, and we suppose that all of these points are jump discontinuities. Let us define $a_n:=\lim_{t\nearrow x_n}{g(t)}$ and $b_n:=\lim_{t\searrow x_n}{g(t)}$. We observe that $f\in\dom((A_{-1}+B)_{|X})$ if and only if
}
\[
\displaystyle{\lim_{t\nearrow x_n}{f'(t)+\Phi(f)a_n}=\lim_{t\searrow x_n}{f'(t)+\Phi(f)b_n}},\quad \text{for each }n\in\NN,
\]
or equivalently
\[
\displaystyle{\lim_{t\nearrow x_n}{f'(t)}-\lim_{t\searrow x_n}{f'(t)}}=\Phi(f)(b_n-a_n),\quad\text{for each } n\in\NN.
\]
We conclude that the operator $(C,\dom(C))$ given by
\begin{align*}
	Cf&=f'+\int_{\RR}{f\ \dd\mu}\cdot g,\\
	\dom(C)&=\left\{f\in\BC(\RR):\ f\in\BC^1(\RR\setminus Z), \displaystyle{\lim_{t\nearrow x_n}{f'(t)}-\lim_{t\searrow x_n}{f'(t)}}=\Phi(f)(b_n-a_n),\quad n\in\NN  \right\}
	\end{align*}
generates a $\tau_{\mathrm{co}}$-bi-continuous semigroup on $\BC(\RR)$.

\section{The left implemented semigroup}\label{sec:impl}

Let $(T(t))_{t\geq0}$ be a $C_0$-semigroup on a Banach space $E$ with generator $(A,\dom(A))$. For simplicity we assume that the growth bound of $(T(t))_{t\geq0}$ satisfies $\omega_0(T)<0$. The left implemented semigroup $(\mathcal{U}(t))_{t\geq0}$ on $\LLL(E)$ is defined by
\begin{align}
\mathcal{U}(t)S:=T(t)S,\quad t\geq0,\ S\in\LLL(E).
\end{align}
Our purpose is to relate Desch--Schappacher perturbations of the $C_0$-semigroup $(T(t))_{t\geq0}$ and Desch--Schappacher perturbations of the left implemented semigroup $(\mathcal{U}(t))_{t\geq0}$.

\begin{remark}
\begin{iiv}
	\item We observe that the left implemented semigroup on $\LLL(E)$ is $\tau_{\sot}$-bi-continuous, where the strong operator topology $\tau_{\sot}$ is induced by the family of seminorms defined by $\semis=\left\{\left\|\cdot~x\right\|_E:\ x\in E\right\}$. In fact $(\mathcal{U}(t))_{t\geq0}$ is a $C_0$-semigroup if and only if the semigroup $(T(t))_{t\geq0}$ is continuous with respect to the operator norm (see \cite{Alber2001} and \cite{KuPhD}).
	\item The extrapolation spaces of the underlying $C_0$-semigroup (in the notation of this paper $\xX_{-n}$, $n\in\NN$) are studied in detail in \cite{Alber2001}. The extrapolation spaces in the bi-continuous setting are determined in \cite{BF}. In particular, the first extrapolation spaces are given by
	\begin{align*}
	\xX_{-1}&=\overline{\LLL(E)}^{\LLL(E,E_{-1})},\\
	X_{-1}&=\overline{\LLL(E)}^{\LLL_{\sot}(E,E_{-1})}=\LLL(E,E_{-1}),\
	\end{align*}
	where $E_{-1}$ is the extrapolation space of the $C_0$-semigroup $(T(t))_{t\geq0}$.
\end{iiv}
\end{remark}

\subsection{Ideals in $\LLL(E)$ and module homomorphisms}

Before we  relate Desch--Schappacher perturbations of $C_0$-semigroups and of the corresponding implemented semigroups we need some auxiliary results. 

\begin{lemma}\label{lem:Ideal}
Let $E$ be a Banach space and $(\A,\dom(\A))$ a Hille--Yosida operator on $\LLL(E)$, i.e., suppose there exists $\omega\in\RR$ and $M\geq1$ such that $(\omega,\infty)\subseteq\rho(\A)$ and
\[
\left\|R(\lambda,\A)^n\right\|\leq\frac{M}{(\lambda-\omega)^n},
\]
for each $\lambda>\omega$ and $n\in\NN$. The following are equivalent:
\begin{iiv}
	\item $\dom(\A)$ is a right ideal of the Banach algebra $\LLL(E)$, i.e., $CB\in\dom(\A)$ whenever $C\in\dom(\A)$, $B\in\LLL(E)$, and $\mathcal{A}$ is a right $\LLL(E)$-module homomorphism, i.e., $\A(CB)=\A(C)B$ for $C\in\dom(\A)$ and $B\in\LLL(E)$.
	\item There exists a Hille--Yosida operator $(A,\dom(A))$ such that $\A(C)=AC$, where $\dom(\A)=\LLL(E,\dom(A))$.
	\item There exists a Hille--Yosida operator $(A,\dom(A))$ such that $\A(C)=A_{-1}C$, where $\dom(\A)=\left\{C\in\LLL(E):\ A_{-1}C\in\LLL(E)\right\}$.
\end{iiv}
\end{lemma}

\begin{proof}
The implication (iii)$\Rightarrow$(i) is just a checking of properties of an explicitly given operator. The implication (ii)$\Leftrightarrow$(iii) follows from the fact that the operator $A$ and $A_{-1}$ coincide on the domain $\dom(A)$ of $A$.

\medskip\noindent  (i)$\Rightarrow$(ii) By definition one has $R(\lambda,\A)\in\LLL(\LLL(E))$ whenever $\lambda\in\rho(\A)$. Define for $\lambda\in\rho(\A)$
\[
R(\lambda):=R(\lambda,\A)(\Id).
\]
Since  $R(\lambda,\A)$, $\lambda\in\rho(\A)$ satisfy the resolvent identity also $R(\lambda)$, $\lambda\in\rho(\A)$ do:
\begin{align*}
R(\lambda)-R(\mu)&=R(\lambda,\A)(\Id)-R(\mu,\A)(\Id)=\left(R(\lambda,\A)-R(\mu,\A)\right)(\Id)\\
&=\left((\lambda-\mu)R(\lambda,\A)R(\mu,\A)\right)(\Id)=(\lambda-\mu)R(\lambda)R(\mu)
\end{align*}
for each $\lambda,\mu\in\rho(\A)$. Hence the family $(R(\lambda))_{\lambda\in\rho(\A)}$ is a pseudoresolvent. If $R(\lambda)x=0$ for some $x\in E$, then
\[
0=\lambda R(\lambda)x=\lambda R(\lambda,\A)(\Id)x.
\]
But since $\lambda R(\lambda,\A)(\Id)\rightarrow\Id$ as $\lambda\to\infty$ , it follow that $x=0$. Therefore $R(\lambda)$ is injective, and hence there exists a closed operator $(A,\dom(A))$ such that $R(\lambda)=R(\lambda,A)$, i.e.,
\[
R(\lambda,A)=R(\lambda,\A)(\Id),
\]
Let $C\in\dom(\A)$, i.e., $C=R(\lambda,\A)D$ for some $D\in\LLL(E)$. Then
\begin{align*}
\A(C)&=\A(R(\lambda,\A)D)=\lambda R(\lambda,\A)D-D=(\lambda R(\lambda,A)-\Id)D\\
&=(\lambda R(\lambda,A)-(\lambda-A)R(\lambda,A))D=AR(\lambda,A)D=AC.
\end{align*}
\end{proof}

\begin{lemma}\label{lem:aux}
Let $E$ be a Banach space and $(\A,\dom(\A))$ a generator of a $\tau_{\sot}$-bi-continuous semigroup $(\T(t))_{t\geq0}$ on $\LLL(E)$. The following are equivalent:
\begin{iiv}
	\item $\dom(\A)$ is a right-ideal of $\LLL(E)$ and $\mathcal{A}$ is a right $\LLL(E)$-module homomorphism.
	\item The semigroup $(\T(t))_{t\geq0}$ is left implemented, i.e., there exists a $C_0$-semigroup $(S(t))_{t\geq0}$ such that $\T(t)C=S(t)C$ for each $t\geq0$.
\end{iiv}
Under these equivalent conditions, if $(B,\dom(B))$ is the generator of the $C_0$-semigroup $(S(t))_{t\geq0}$, then $\A(C)=B_{-1}C$ for each $C\in\LLL(E,E_{-1})=X_{-1}(\A)$.
\end{lemma}

\begin{proof}
(ii)$\Rightarrow\mathrm{(i)}:$ If $C\in\dom(\A)$, then the limit
\[
(\A C)(x):=\lim_{t\searrow0}{\frac{\T(t)Cx-Cx}{t}},
\]
exists for each $x\in X$. Since $(\T(t))_{t\geq0}$ is left implemented we obtain for $B\in\LLL(E)$
\[
(\A(CB))(x)=\lim_{t\rightarrow0}{\frac{\T(t)(CB)x-(CB)x}{t}}=\lim_{t\rightarrow0}{\frac{(\T(t)C)(Bx)-C(Bx)}{t}},
\]
and we conclude that $CB\in\dom(\A)$ and $\A(CB)=\A(C)B$.

\medskip\noindent(i)$\Rightarrow{\mathrm{(ii)}}:$ For $\lambda\in\rho(\A)$, $C\in\dom(\A)$, $B\in\LLL(E)$ one has
\[
(\lambda-\A)(CB)=\lambda CB-\A(C)B=(\lambda C-\A(C))B.
\]
Since $\lambda-\A$ is a bijective map we conclude that
\[
R(\lambda,\A)(DB)=(R(\lambda,\A)D)B
\]
for each $D\in\LLL(E)$. By the Euler-Formula (see \cite[Thm. 4.6]{BF} and \cite[Chapter II, Sect. 3]{EN}) we obtain
\[
\T(t)C=\tsotlim_{n\rightarrow\infty}\left(\frac{n}{t}R\left(\frac{n}{t},\A\right)\right)^nC.
\]
From this we deduce the equality
\[
\T(t)(CB)(x)=\left(\tlim_{n\rightarrow\infty}\left(\frac{n}{t}R\left(\frac{n}{t},\A\right)\right)^nC\right)B=(\T(t)C)B.
\]
Set $S(t):=\T(t)\Id$, and we are done
\[
\T(t)C=\T(t)(\Id\cdot C)=(\T(t)\Id)C=S(t)C.
\]
Finally, $\A$ is multiplication operator by the generator $(B,\dom(B))$ of the semigroup $(S(t))_{t\geq0}$ by Lemma \ref{lem:Ideal}.
\end{proof}

\begin{proposition}\label{prop:DSMult}
	Let $(T(t))_{t\geq0}$ and $(S(t))_{t\geq0}$ be $C_0$-semigroups on the Banach space $E$, and let   $(A,\dom(A))$ denote the generator of $(T(t))_{t\geq0}$.
	Let $(\mathcal{U}(t))_{t\geq0}$ and $(\mathcal{V}(t))_{t\geq0}$ be the semigroups left implemented by $(T(t))_{t\geq0}$ and $(S(t))_{t\geq0}$, respectively. Let $(\mathcal{G},\dom(\mathcal{G}))$ be the generator of $(\mathcal{U}(t))_{t\geq0}$ and let $\mathcal{K}:\LLL(E)\rightarrow\LLL(E,E_{-1}(A))$ be such that $\mathcal{K}\in\mathcal{S}^{DS,\tau_{\sot}}_{t_0}(\mathcal{U})$ and such that $\mathcal{C}:=(\mathcal{G}_{-1}+\mathcal{K})_{|\LLL(E)}$ (with maximal domain) is the generator of $(\mathcal{V}(t))_{t\geq0}$. Then $\mathcal{K}$ has the property that 
\[
\mathcal{K}(CD)=\mathcal{K}(C)D,
\]
for each $C,D\in\LLL(E)$.
\end{proposition}

\begin{proof}
Since by assumption $\mathcal{G}$ and $\mathcal{C}=(\mathcal{G}_{-1}+\mathcal{K})_{|\LLL(E)}$ both generate implemented semigroups we conclude by Lemma \ref{lem:aux} that $\mathcal{G}$, and hence $\mathcal{G}_{-1}$, and $\mathcal{C}$ are all multiplication operators. One has $\mathcal{G}_{-1}(C)=A_{-1}C$ for each $C\in\LLL(E)$ and there exists an operator $M:E\rightarrow E_{-1}(L)$ such that $\mathcal{C}(C)=MC$ for each $C\in\dom(\mathcal{C})$. We conclude that
\[
\mathcal{K}(C)=MC-A_{-1}C
\]
for each $C\in\dom(\mathcal{C})$. Since $(\mathcal{C},\dom(\mathcal{C}))$ is bi-dense in $\LLL(E)$, for each $C\in\LLL(E)$, there exists a sequence of operators $(C_n)_{n\in\NN}$ in $\dom(\mathcal{C})$ such that $\sup_{n\in\NN}{\left\|C_n\right\|}<\infty$ and
\[
C_nx\rightarrow Cx,
\]
for each $x\in E$. The continuity of $\mathcal{K}$ and $\mathcal{G}_{-1}$ yields
\begin{align*}
&\mathcal{K}(C_n)\stackrel{\tau_{\sot}}{\rightarrow}\mathcal{K}(C),\\
&\mathcal{G}(C_n)\stackrel{\tau_{\sot}}{\rightarrow}\mathcal{G}(C)
\end{align*}
with convergence in $\LLL_{\sot}(E,E_{-1}(A))$. 
Therefore, for each $x\in E$ the sequence $(MC_nx)_{n\in\NN}$ is Cauchy in $E_{-1}(A)$ and we can define
\[
Lx:=\lim_{n\rightarrow\infty}{MC_nx},\quad x\in E.
\]
By construction we obtain $L\in\LLL(E,E_{-1}(A))$ and $\mathcal{C}(C)=LC$ for each $C\in\LLL(E)$ and therefore
\[
\mathcal{K}(C)=A_{-1}C+LC,
\]
for $C\in\dom(\mathcal{C})$. Now we define $B:=L+A_{-1}$ as an operator in $\LLL(E,E_{-1}(A))$ and conclude that
\[
\mathcal{K}(C)=BC
\]
for each $C\in\LLL(E)$ and that was to be proven.
\end{proof}

\subsection{A one-to-one correspondence}

We relate Desch--Schappacher perturbations of the implemented semigroup with the perturbations of the underlying $C_0$-semigroup. To do so we have to use the class of Desch--Schappacher admissible operators $\mathcal{S}_{t_0}^{DS}$ for $C_0$-semigroups. Recall from \cite[Chapter III, Section 3a]{EN}  the following definitions for a strongly continuous semigroup $(T(t))_{t\geq0}$ on a Banach space $E$. We define 
\[
\mathcal{S}_{t_0}^{DS}(T):=\left\{B\in\LLL(E,E_{-1}):\ V_B\in\LLL\left(\mathrm{C}\left(\left[0,t_0\right],\LLL_{\sot}(E)\right)\right),\ \left\|V_B\right\|<1\right\},
\]
where $V_B$ denotes the corresponding Volterra operator on $E$ defined by
\[
(V_BF)(t):=\int_0^{t}{T_{-1}(t-r)BF(r)\ \dd{r}},\quad F\in\mathrm{C}\left(\left[0,t_0\right],E\right),\ t\in\left[0,t_0\right].
\]
The following result shows that Desch--Schappacher perturbations of a $C_0$-semigroup always give us Desch--Schappacher perturbations of the corresponding implemented semigroup.

\begin{theorem}\label{thm:Impl1}  
Let $(\mathcal{U}(t))_{t\geq0}$ be the semigroup on $\LLL(E)$ left implemented by the $C_0$-semigroup $(T(t))_{t\geq0}$. Suppose that $B\in\mathcal{S}_{t_0}^{DS}$ and let $(S(t))_{t\geq0}$ be the perturbed $C_0$-semigroup. Define the operator $\mathcal{K}:\LLL(E)\rightarrow\LLL(E,E_{-1})$ by
\[
\mathcal{K}S:=BS,\quad S\in\LLL(E).
\]
Then $\mathcal{K}\in\mathcal{S}^{DS,\tau_{\mathrm{sot}}}_{t_0}$ and the perturbed semigroup $(\mathcal{V}(t))_{t\geq0}$ is left implemented by $(S(t))_{t\geq0}$.
\end{theorem}

\begin{proof}
 First of all we show that $V_{\mathcal{K}}F(t)C\in\LLL(E)$ for $F\in\mathfrak{X}_{t_0}$, $t\in\left[0,t_0\right]$ and $C\in\LLL(E)$. Define $f\in\mathrm{C}\left(\left[0,t_0\right],\LLL_{\sot}(E)\right)$ by $f(r):=F(r)C$ and observe
\begin{align*}
(V_{\mathcal{K}}F)(t)Cx=\int_0^{t}{\mathcal{U}_{-1}(t-r)\mathcal{K}F(r)Cx\ \dd{r}}=\int_0^t{T_{-1}(t-r)Bf(r)x\ \dd{r}}.
\end{align*}
Since by assumption $B\in\mathcal{S}^{DS}_{t_0}$, we obtain $(V_{\mathcal{K}}F)(t)Cx\in E$. The following estimate will be crucial for what follows
\begin{align*}
\left\|(V_{\mathcal{K}}F)(t)Cx\right\|&=\left\|\int_0^t{\mathcal{U}_{-1}(t-r)\mathcal{K}F(r)Cx\ \dd{r}}\right\|=\left\|\int_0^t{T_{-1}(t-r)BF(r)Cx}\right\|\\
&=\left\|(V_Bf)(t)x\right\|\leq \left\|V_B\right\|\cdot\left\|f\right\|\cdot\left\|Cx\right\|\leq \left\|V_B\right\|\cdot\left\|f\right\|\cdot\left\|C\right\|\cdot\left\|x\right\|.
\end{align*}
This estimate shows that $(V_{\mathcal{K}}F)(t)C\in\LLL(E)$. Moreover, we directly see that $\mathrm{Ran}(V_{\mathcal{K}})\subseteq\mathfrak{X}_{t_0}$, since $\tau_{\sot}$-strong continuity, norm boundedness and bi-equicontinuity of $V_{\mathcal{K}}F$  follow also from the previous estimate. Also the fact that $\left\|V_{\mathcal{K}}\right\|<1$ is immediate, due to the assumption that $B\in\mathcal{S}^{DS}_{t_0}$.

Finally, we show that $(\mathcal{G}_{-1}+\mathcal{K})_{|\LLL(E)}$ generates the semigroup left implemented by $(S(t))_{t\geq0}$. For this notice that for sufficiently large $\lambda>0$ we have
\begin{align*}
R(\lambda,(A_{-1}+B)_{E})Cx&=\int_0^{\infty}{\ee^{-\lambda t}S(t)Cx\ \dd t}=\int_0^{\infty}{\ee^{-\lambda t}\mathcal{V}(t)Cx\ \dd t}\\
&=R(\lambda,(\mathcal{G}_{-1}+\mathcal{K})_{|\LLL(E)})Cx,
\end{align*}
for all $x\in E$ and $C\in \LLL(E)$. Whence we conclude that $(\mathcal{G}_{-1}+\mathcal{K})_{|\LLL(E)}$ generates the semigroup left implemented by $(S(t))_{t\geq0}$.
\end{proof}

Here is the converse of this theorem.

\begin{theorem}\label{thm:ImplC0}
Let $(\mathcal{U}(t))_{t\geq0}$ and $(\mathcal{V}(t))_{t\geq0}$ be two semigroups on $\LLL(E)$, left implemented by the $C_0$-semigroups $(T(t))_{t\geq0}$ and $(S(t))_{t\geq0}$, respectively. Let $(A,\dom(A))$ be the generator of  $(T(t))_{t\geq0}$ and let $\mathcal{K}\in\mathcal{S}^{DS,\tau_{\sot}}_{t_0}(\mathcal{U})$ be such that $(\mathcal{V}(t))_{t\geq0}$ is the corresponding perturbed semigroup. Define $B\in\LLL(E,E_{-1})$ by
\[
Bx:=(\mathcal{K}\Id)x,\quad x\in E.
\]
Then $B\in\mathcal{S}^{DS}_{t_0}(T)$ and $(A_{-1}+B)_{|E}$ generates $(S(t))_{t\geq0}$.
\end{theorem}

\begin{proof}
 Let $f\in\mathrm{C}\left(\left[0,t_0\right],\LLL_{\sot}(E)\right)$ and $x\in E$. We observe that by Lemma \ref{lem:Ideal} one has that $(\mathcal{K}\Id)f(r)=\mathcal{K}(\Id f(r))=\mathcal{K}f(r)$ for each $r\in\left[0,t_0\right]$. For $f\in\mathrm{C}\left(\left[0,t_0\right],\LLL_{\sot}(E)\right)$ we define $F\in\mathfrak{X}_{t_0}$ by $F(r):=M_{f(r)}$, the multiplication with $f(r)$, i.e., $F(r)C=f(r)C$ for each $C\in\LLL(E)$. The following computation is crucial for the proof:
\begin{align*}
V_Bf(t)x&=\int_0^t{T_{-1}(t-r)Bf(r)x\ \dd{r}}=\int_0^{t}{T_{-1}(t-r)(\mathcal{K}\Id)f(r)x\ \dd{r}}\\
&=\int_0^t{\mathcal{U}_{-1}(t-r)\mathcal{K}f(r)x\ \dd{r}}=\int_0^t{\mathcal{U}_{-1}(t-r)\mathcal{K}F(r)\Id x\ \dd{r}}\\
&=(V_{\mathcal{K}}F)(t)\Id x.
\end{align*}
From this and from the assumption that $\mathcal{K}\in \mathcal{S}^{DS,\tau_{\sot}}_{t_0}$, we conclude that $B\in\mathcal{S}^{DS}_{t_0}$.
Moreover we have
\[
S(t)x=\mathcal{V}(t)\Id x=\mathcal{U}(t)\Id x+\int_0^t{\mathcal{U}_{-1}(t-r)\mathcal{K}\mathcal{V}(r)\Id x\ \dd{r}}=T(t)x+\int_0^t{T_{-1}(t-r)BS(r)x\ \dd{r}},
\]
for each $x\in E$. This yields that $(A_{-1}+B)_{|E}$ generates $(S(t))_{t\geq0}$.
\end{proof}

Summarizing Theorems \ref{thm:ImplC0} and  \ref{thm:Impl1} we  can state the following.

\begin{corollary}
Let $(\mathcal{U}(t))_{t\geq0}$ and $(\mathcal{V}(t))_{t\geq0}$ be two semigroups on $\LLL(E)$ left implemented by the $C_0$-semigroups $(T(t))_{t\geq0}$ and $(S(t))_{t\geq0}$ on $E$, respectively. Let us denote the generators of $(\mathcal{U}(t))_{t\geq0}$ and $(T(t))_{t\geq0}$ by $(\mathcal{G},\dom(\mathcal{G}))$ and $(A,\dom(A))$, respectively. The following are equivalent:
\begin{iiv}
	\item There exists $\mathcal{K}\in\mathcal{S}^{DS,\tau}_{t_0}(\mathcal{U})$ such that $(\mathcal{V}(t))_{t\geq0}$ is generated by $(\mathcal{G}_{-1}+\mathcal{K})_{|\LLL(E)}$.
	\item There exists $B\in\mathcal{S}^{DS}_{t_0}(T)$ such that $(S(t))_{t\geq0}$ is generated by $(A_{-1}+B)_{|E}$.
\end{iiv}
\end{corollary}

\begin{remark}
Notice that not every Desch--Schappacher perturbation of a implemented semigroup gives again a implemented semigroup. To see this let $(\mathcal{G},\dom(\mathcal{G}))$ be the generator of the left implemented semigroup $(\mathcal{U}(t))_{t\geq0}$ and $\Phi\in(\LLL(E),\tau_{\sot})'$. Define, as above, an operator $\mathcal{K}:\LLL(E)\rightarrow\LLL(E,E_{-1})$ by
\[
\mathcal{K}(C):=\Phi(C)\mathcal{G}_{-1}(\Id),\quad C\in\LLL(E).
\]
Such an operator $\mathcal{K}$ is not multiplicative if $\Phi\neq0$.
\end{remark}

\subsection{Comparisons}

Now we relate comparison properties of the implemented semigroup and properties of the underlying $C_0$-semigroup. First of all, for $B\in\LLL(E)$ we define the multiplication operator $M_B\in\LLL(\LLL(E),\LLL(E))$ by $M_BS:=BS$. Then one has $\left\|M_B\right\|=\left\|B\right\|$. By taking $B:=T(t)-S(t)$ for $t>0$ we directly obtain the following result. 

\begin{lemma}\label{lem:CompImpl}
Let $(\mathcal{U}(t))_{t\geq0}$ and $(\mathcal{V}(t))_{t\geq0}$ be two semigroups on $\LLL(E)$ left implemented by the $C_0$-semigroups $(T(t))_{t\geq0}$ and $(S(t))_{t\geq0}$, respectively. Then the following are equivalent:
\begin{iiv}
	\item There exists $M\geq0$ such that $\left\|\mathcal{U}(t)-\mathcal{V}(t)\right\|\leq Mt$ for each $t\in\left[0,1\right]$.
	\item There exists $M\geq0$ such that $\left\|T(t)-S(t)\right\|\leq Mt$ for each $t\in\left[0,1\right]$.
\end{iiv}
\end{lemma}

Recall from \cite[Prop. 6.1]{BF} that the Favard spaces of the implemented semigroup and of the underlying semigroup for $\alpha\in\left[0,1\right]$ are connected by
\begin{align}\label{eqn:FavImpl}
F_{\alpha}(\mathcal{U})=\LLL(E,F_{\alpha}(T)).
\end{align}

This yields to the following result.

\begin{lemma}\label{lem:FavImpl}
For $B\in\LLL(E,E_{-1})$ we define $\mathcal{K}:\LLL(E)\rightarrow\LLL(E,E_{-1})$ by $\mathcal{K}S:=BS$. Then $\mathrm{Ran}(\mathcal{K})\subseteq F_{\alpha}(\mathcal{U})$ if and only if $\mathrm{Ran}(B)\subseteq F_{\alpha}(T)$.
\end{lemma}

By \cite{Alber2001} and \cite{BF} the extrapolated implemented semigroup is defined by
\[
\mathcal{U}_{-1}(t)S=T_{-1}(t)S,\quad S\in\overline{\LLL(E)}^{\LLL_{\sot}(E,E_{-1})}=\LLL(E,E_{-1}).
\]
This gives
\[
F_0(\mathcal{U})=F_1(\mathcal{U}_{-1})=\LLL(E,F_1(T_{-1}))=\LLL(E,F_0(T)).
\]

\begin{proposition}\label{prop:ImplDS1}
Let $(\mathcal{U}(t))_{t\geq0}$ and $(\mathcal{V}(t))_{t\geq0}$ be two semigroups on $\LLL(E)$ left implemented by $(T(t))_{t\geq0}$ and $(S(t))_{t\geq0}$, respectively. Furthermore let $(\mathcal{G},\dom(\mathcal{G}))$ denote the generator of $(\mathcal{U}(t))_{t\geq0}$. Suppose that there exists $M\geq0$ such that 
\[
\left\|\mathcal{U}(t)-\mathcal{V}(t)\right\|\leq Mt
\]
for each $t\in\left[0,1\right]$. Then there exists $\mathcal{K}\in\mathcal{S}_{t_0}^{DS,\tau}$ with $\mathrm{Ran}(\mathcal{K})\subseteq F_0(\mathcal{G})$.
\end{proposition}

\begin{proof}
Since $\left\|\mathcal{U}(t)-\mathcal{V}(t)\right\|\leq Mt$ for each $t\in\left[0,1\right]$ we can use Lemma \ref{lem:CompImpl} to conclude that $\left\|T(t)-S(t)\right\|\leq Mt$ for each $t\in\left[0,1\right]$. If $(A,\dom(A))$ denotes the generator of $(T(t))_{t\geq0}$, then by \cite[Chapter III, Thm. 3.9]{EN} we find $B\in\LLL(E,E_{-1})$ such that $B\in\mathcal{S}^{DS}_{t_0}$ and $\mathrm{Ran}(B)\subseteq F_0(A)$. As in Theorem \ref{thm:Impl1} this gives rise to an multiplication operator $\mathcal{K}:\LLL(E)\rightarrow\LLL(E,E_{-1})$ defined by
\[
\mathcal{K}S:=BS,\quad S\in\LLL(E).
\]
By Lemma \ref{lem:FavImpl} we conclude that $\mathrm{Ran}(\mathcal{K})\subseteq F_0(\mathcal{G})$. It remains to show that $(\mathcal{G}_{-1}+\mathcal{K})_{|\LLL(E)}$ generates $(\mathcal{V}(t))_{t\geq0}$. But, by \cite[Chapter III, Thm. 3.9]{EN}, $(A_{-1}+B)_{|E}$ generates $(S(t))_{t\geq0}$.
\end{proof}

Combining Propositions \ref{prop:ImplDS1}, \ref{prop:ImplDS2} and \cite[Chapter III, Thm. 3.9]{EN} we obtain the following theorem.

\begin{theorem}\label{thm:Impl2}
Let $(\mathcal{U}(t))_{t\geq0}$ and $(\mathcal{V}(t))_{t\geq0}$ be two semigroups on $\LLL(E)$ left implemented by $(T(t))_{t\geq0}$ and $(S(t))_{t\geq0}$, respectively. Denote by $(\mathcal{G},\dom(\mathcal{G}))$ the generator of $(\mathcal{U}(t))_{t\geq0}$ and by $(A,\dom(A))$ the generator of $(T(t))_{t\geq0}$. If $\mathcal{K}\in\mathcal{S}^{DS,\tau}_{t_0}(\mathcal{U})$ such that $\mathrm{Ran}(\mathcal{K})\subseteq F_0(\mathcal{G})$, then there exists $B\in\mathcal{S}_{t_0}^{DS}(T)$ with $\mathrm{Ran}(B)\subseteq F_0(A)$ such that $\mathcal{K}S=BS$ for each $S\in\LLL(E)$.
\end{theorem}

\begin{proof}
By Proposition \ref{prop:ImplDS2} we find $M\geq0$ such that $\left\|\mathcal{U}(t)-\mathcal{V}(t)\right\|\leq Mt$ for each $t\in\left[0,1\right]$. Following the proof of Proposition \ref{prop:ImplDS1} there exists $B\in\mathcal{S}^{DS}_{t_0}$ such that $\mathrm{Ran}(B)\subseteq F_0(A)$.
\end{proof}

From this we can deduce the following equivalence. 

\begin{theorem}\label{thm:CompEquiv}
Let $(\mathcal{U}(t))_{t\geq0}$ and $(\mathcal{V}(t))_{t\geq0}$ be two semigroups on $\LLL(E)$ left implemented by the $C_0$-semigroups $(T(t))_{t\geq0}$ and $(S(t))_{t\geq0}$ on $E$, respectively. Let us denote the generators of $(\mathcal{U}(t))_{t\geq0}$ and $(T(t))_{t\geq0}$ by $(\mathcal{G},\dom(\mathcal{G}))$ and $(A,\dom(A))$, respectively. The following are equivalent:
\begin{iiv}
	\item There exists $\mathcal{K}\in\mathcal{S}^{DS,\tau}_{t_0}(\mathcal{U})$ such that $\mathrm{Ran}(\mathcal{K})\subseteq F_0(\mathcal{G})$ and such that $(\mathcal{V}(t))_{t\geq0}$ is generated by $(\mathcal{G}_{-1}+\mathcal{K})_{|\LLL(E)}$.
	\item There exists $B\in\mathcal{S}^{DS}_{t_0}(T)$ such that $\mathrm{Ran}(B)\subseteq F_0(A)$ and such that $(S(t))_{t\geq0}$ is generated by $(A_{-1}+B)_{|E}$.
\end{iiv}
\end{theorem}


\providecommand{\bysame}{\leavevmode\hbox to3em{\hrulefill}\thinspace}
\providecommand{\MR}{\relax\ifhmode\unskip\space\fi MR }
\providecommand{\MRhref}[2]{%
  \href{http://www.ams.org/mathscinet-getitem?mr=#1}{#2}
}
\providecommand{\href}[2]{#2}

\end{document}